\documentclass[11pt]{article}
\usepackage[a4paper,margin=2.7cm]{geometry}
\usepackage[T1]{fontenc}
\usepackage[utf8]{inputenc}
\usepackage{lmodern}
\usepackage{hyperref}

\usepackage{amsmath,amssymb,amsfonts}
\usepackage{amsthm}
\usepackage{mathrsfs}

\theoremstyle{plain}
\newtheorem{theorem}{Theorem}[section]
\newtheorem{lemma}[theorem]{Lemma}
\newtheorem{proposition}[theorem]{Proposition}
\newtheorem{corollary}[theorem]{Corollary}
\newtheorem{problem}[theorem]{Open Problem}

\theoremstyle{definition}

\newtheorem{remark}[theorem]{Remark}

\theoremstyle{definition}
\newtheorem{definition}[theorem]{Definition}

\begin{document}

\title{Order Boundedness of Order Continuous Orthogonally Additive Operators} 

\author{Ayşe Uyar\thanks{Email: ayseu@gazi.edu.tr}\\Retired Professor of Mathematics, Istanbul, Türkiye}
\date{}

\maketitle

\begin{abstract}
We prove that, under suitable assumptions, the space of order
continuous orthogonally additive operators between Riesz spaces is a
Riesz space. As a consequence, we give a negative answer to an open
problem posed by Mykhaylyuk and Popov concerning the order continuity
of the modulus of an order continuous orthogonally additive operator.
We also show that, under an appropriate condition on the domain,
every order continuous orthogonally additive operator is order bounded,
obtaining an orthogonally additive analogue of the corresponding
result for linear operators.
Moreover, we establish the order boundedness result
in ZF for atomic Riesz spaces and in ZF with
dependent choice (DC) for Dedekind complete Riesz spaces.
Finally, we characterize when every
orthogonally additive operator is an abstract Uryson operator.
\end{abstract}



\section{Introduction}

Orthogonally additive operators form a natural nonlinear extension of
linear operators on Riesz spaces. A mapping $T:X\to Y$ between Riesz
spaces is called orthogonally additive if $T(x+y)=T(x)+T(y)$ whenever
$x\perp y$. Among the most important classes of such operators are
abstract Uryson operators, that is, order bounded orthogonally additive
operators, and Popov operators, which are orthogonally additive
operators that are laterally-to-order bounded. Orthogonally additive
operators and various aspects of their theory have been studied
extensively; see, among others,
\cite{FotiyKadetsPopov23,
FotiyKrasikovaPlievPopov22,
MazonSegura90,
MykhaylyukPlievPopov21,
MykhaylyukPlievPopovSobchuk,
MykhaylyukPopov22,
PlievRamdane18}.

Order continuity of orthogonally additive operators was studied by
Fotiy, Krasikova, Pliev and Popov
\cite{FotiyKrasikovaPlievPopov22} and subsequently by Mykhaylyuk and
Popov \cite{MykhaylyukPopov22}. In the latter paper, the following
problem was posed.

\begin{problem}[{\cite[Problem~4]{MykhaylyukPopov22}}]
Let $X$ be a Riesz space with the principal projection property
(or even a Dedekind complete Riesz space), and let $Y$ be a
Dedekind complete Riesz space. Does there exist an order continuous
orthogonally additive operator $T:X\to Y$ such that $|T|$ is not
order continuous?
\end{problem}

To the best of our knowledge, this problem has remained open
(M.~Popov, private communication). In this paper, we obtain three
main results.

First, we prove that, under the assumptions above, the space
$\mathcal O_o(X,Y)$ of all order continuous orthogonally additive
operators is a Riesz space. In particular, the modulus of every
order continuous orthogonally additive operator is order continuous,
which gives a negative answer to the above problem.

Second, we prove that every order continuous orthogonally additive
operator from a Riesz space $X$ with a cofinal family of projection
bands into an arbitrary Riesz space $Y$ is order bounded, and hence
is an abstract Uryson operator. This result may be viewed as an
orthogonally additive counterpart of a classical result of Abramovich
and Sirotkin \cite[Theorem~2.1]{AbramovichSirotkin05}, which states
that every order continuous linear operator between Riesz spaces is
automatically order bounded. Although orthogonally additive operators
form a natural nonlinear extension of linear operators, the
corresponding implication is no longer automatic, since the linearity
argument is not available. Our result shows that the implication
remains valid for orthogonally additive operators when the domain has
a cofinal family of projection bands.

The assumption that $X$ has a cofinal family of projection bands is
weaker than the principal projection property. In particular, our
order boundedness result applies whenever $X$ has the principal
projection property, and hence whenever $X$ is Dedekind complete.

We also examine the role of the axiom of choice
in the order boundedness result, showing that
the conclusion can be established in
Zermelo--Fraenkel set theory (ZF) for atomic
Riesz spaces and in ZF together with the
axiom of dependent choice (DC) for
Dedekind complete Riesz spaces.

Finally, if $X$ has a cofinal family of projection bands and $Y$ is
Archimedean, we characterize when every orthogonally additive operator
is an abstract Uryson operator. More precisely, we prove that
\[
\mathcal O(X,Y)=\mathcal U(X,Y)
\]
if and only if $X=\{0\}$ or $Y=\{0\}$.

\section{Preliminaries}

Throughout this paper, all Riesz spaces are assumed to be Archimedean.
We use the standard terminology and notation of Riesz space theory as
in \cite{AliprantisBurkinshaw06,LuxemburgZaanen71,Zaanen83}.

\begin{definition}
A Riesz space $X$ is said to have a \emph{cofinal family of projection
bands} if, for every nonzero band $B$ in $X$, there exists a nonzero
projection band $D$ such that
\[
D\subseteq B.
\]
\end{definition}

This property is weaker than the principal projection property; see
Abramovich and Kitover \cite[Definition~4.2]{AbramovichKitover00}.
In particular,
\[
\text{Dedekind complete}
\quad\Longrightarrow\quad
\text{principal projection property}
\quad\Longrightarrow\quad
\text{cofinal family of projection bands}.
\]

Let $X$ be a Riesz space. Two elements $x,y\in X$ are said to be
\emph{disjoint}, written $x\perp y$, if $|x|\wedge |y|=0$. An element
$y\in X$ is called a \emph{fragment} (or \emph{component}) of
$x\in X$, written $y\sqsubseteq x$, if $y\perp(x-y)$. The set of all
fragments of $x$ is denoted by
\[
\mathfrak F_x:=\{y\in X:y\sqsubseteq x\}.
\]
For $x,y,z\in X$, the notation $x=y\sqcup z$ means that
$x=y+z$ and $y\perp z$.

\begin{definition}
Let $X$ be a Riesz space.

\begin{itemize}
\item[(i)] A nonzero element $u\in X$ is called an
\emph{atom} if $0\le x,y\le |u|$ and
$x\perp y$ imply that $x=0$ or $y=0$.

\item[(ii)] The Riesz space $X$ is called
\emph{atomic} (or \emph{discrete}) if every
nonzero positive element of $X$ dominates
an atom.

\item[(iii)] A nonzero element $u\in X$ is called a
\emph{weak atom} if
$\mathfrak F_u=\{0,u\}$.
\end{itemize}
\end{definition}

Every atom is a weak atom. Moreover, the band
generated by an atom $u$ is a one-dimensional
projection band; see
\cite[Theorem~26.4(iii)]{LuxemburgZaanen71}.
Consequently, every atomic Riesz space has a
cofinal family of projection bands.

If $X$ has the principal projection property,
then every weak atom is an atom; see
\cite[Proposition~1.3]{MykhaylyukPlievPopovSobchuk}.
In Section~3, we show that this conclusion
remains valid when $X$ has a cofinal family
of projection bands.

Throughout the paper, nets and subnets are understood in the sense of
Kelley \cite{Kelley55}.

\begin{definition}
Let $X$ be a Riesz space and $V$ a real vector space. A mapping
$T:X\to V$ is called an \emph{orthogonally additive operator}
(OAO, for short) if $T(x+y)=T(x)+T(y)$ whenever $x\perp y$.
\end{definition}

Every orthogonally additive operator satisfies $T(0)=0$. In particular,
a mapping $T:\mathbb R\to\mathbb R$ is orthogonally additive if and
only if $T(0)=0$. If $X$ and $Y$ are Riesz spaces, we denote by
$O(X,Y)$ the vector space of all orthogonally additive operators from
$X$ to $Y$.

\begin{definition}
Let $X$ and $Y$ be Riesz spaces and let $T\in O(X,Y)$. The operator
$T$ is called
\begin{itemize}
\item \emph{positive} if $T(x)\ge0$ for every $x\in X$;
\item \emph{order bounded} if it maps order bounded subsets of $X$
      to order bounded subsets of $Y$;
\item \emph{laterally-to-order bounded} if $T(\mathfrak F_x)$ is
      order bounded in $Y$ for every $x\in X$.
\end{itemize}
\end{definition}

An order bounded orthogonally additive operator is called an
\emph{abstract Uryson operator}. We denote by $U(X,Y)$ the vector
space of all abstract Uryson operators from $X$ to $Y$. Following
Pliev and Ramdane \cite{PlievRamdane18}, a laterally-to-order bounded
orthogonally additive operator is called a \emph{Popov operator}. We
denote by $P(X,Y)$ the vector space of all Popov operators from $X$ to
$Y$. Clearly,
\[
U(X,Y)\subseteq P(X,Y)\subseteq O(X,Y).
\]

The vector spaces $O(X,Y)$, $P(X,Y)$, and $U(X,Y)$ are equipped with
the natural order: $S\le T$ if and only if $T-S$ is positive. With
this order, they are ordered vector spaces.

The lattice structure of abstract Uryson operators was established by
Maz\'on and Segura de Le\'on.

\begin{theorem}
\cite[Theorem~3.2]{MazonSegura90}
Let $X$ and $Y$ be Riesz spaces, with $Y$ Dedekind complete. Then
$U(X,Y)$ is a Dedekind complete Riesz space. Moreover, for
$S,T\in U(X,Y)$ and $x\in X$,
\[
\begin{aligned}
(T\vee S)(x)
 &=\sup\{T(y)+S(z):x=y\sqcup z\},\\
(T\wedge S)(x)
 &=\inf\{T(y)+S(z):x=y\sqcup z\},\\
T^+(x)
 &=\sup\{T(y):y\sqsubseteq x\},\\
T^-(x)
 &=-\inf\{T(y):y\sqsubseteq x\}.
\end{aligned}
\]
Furthermore, $|T(x)|\le |T|(x)$ for every $x\in X$.
\end{theorem}

Pliev and Ramdane \cite[Theorem~3.6]{PlievRamdane18} proved that the
same result holds for the larger space $P(X,Y)$ of Popov operators:
whenever $Y$ is Dedekind complete, $P(X,Y)$ is a Dedekind complete
Riesz space, and its lattice operations are given by the same formulas
as above.

We next recall the convergence notions used in the paper.

\begin{definition}
Let $(x_\alpha)_{\alpha\in\Gamma}$ be a net in a Riesz space $X$ and
let $x\in X$.
\begin{itemize}
\item[(i)] $(x_\alpha)$ is \emph{strongly order convergent} (or
\emph{so-convergent}) to $x$, written
$x_\alpha\xrightarrow{so}x$, if there exists a net
$(z_\alpha)_{\alpha\in\Gamma}$ in $X$ such that $z_\alpha\downarrow0$
and, for some $\alpha_0\in\Gamma$,
$|x_\alpha-x|\le z_\alpha$ for all $\alpha\ge\alpha_0$.

\item[(ii)] $(x_\alpha)$ is \emph{order convergent} (or
\emph{o-convergent}) to $x$, written $x_\alpha\xrightarrow{o}x$, if
there exists a net $(z_\beta)_{\beta\in\Lambda}$ in $X$ such that
$z_\beta\downarrow0$ and, for every $\beta\in\Lambda$, there exists
$\alpha_0\in\Gamma$ such that $|x_\alpha-x|\le z_\beta$ for all
$\alpha\ge\alpha_0$.

\item[(iii)] $(x_\alpha)$ is said to \emph{laterally converge}
to $x$, written $x_\alpha\xrightarrow{lat}x$, if
\[
x_\alpha\sqsubseteq x_\beta\sqsubseteq x
\quad\text{whenever }\alpha\leq\beta,
\qquad\text{and}\qquad
x_\alpha\xrightarrow{o}x.
\] 
\end{itemize}
\end{definition}

These convergence notions satisfy
\[
x_\alpha\xrightarrow{so}x
\quad\Longrightarrow\quad
x_\alpha\xrightarrow{o}x
\]
If $X$ is Dedekind complete, then strong order convergence and order
convergence coincide; see Abramovich and Sirotkin
\cite{AbramovichSirotkin05}.

\begin{definition}
Let $X$ and $Y$ be Riesz spaces and let $T\in O(X,Y)$.
\begin{itemize}
    \item[(i)] $T$ is called \emph{order continuous} if
    \[
    x_\alpha\xrightarrow{o}x
    \quad\Longrightarrow\quad
    T(x_\alpha)\xrightarrow{o}T(x).
    \]

    \item[(ii)] $T$ is called \emph{laterally continuous} if
    \[
    x_\alpha\xrightarrow{\mathrm{lat}}x
    \quad\Longrightarrow\quad
    T(x_\alpha)\xrightarrow{o}T(x).
    \]
\end{itemize}
\end{definition}

\section{Main Results}

We denote by $\mathcal O_o(X,Y)$ the space of all order continuous
orthogonally additive operators from $X$ to $Y$. We first observe that
\[
\mathcal O_o(X,Y)\subseteq P(X,Y).
\]

\begin{proposition}\label{prop:oc-popov}
Let $X$ and $Y$ be Riesz spaces. If $T:X\to Y$ is an order continuous
orthogonally additive operator, then
\[
T\in P(X,Y).
\]
\end{proposition}

\begin{proof}
Let $(x_\alpha)$ be a net in $X$ which laterally converges to some
$x\in X$. By the definition of lateral convergence,
$x_\alpha\xrightarrow{o}x$. Since $T$ is order continuous,
\[
T(x_\alpha)\xrightarrow{o}T(x).
\]
Thus $T$ is laterally continuous. It follows from
\cite[Lemma~3.12]{PlievRamdane18} that $T$ is laterally-to-order bounded,
and hence $T\in P(X,Y)$.
\end{proof}

Combining the preceding proposition with
\cite[Theorem~6.1]{MykhaylyukPopov22}, we obtain the following consequence.

\begin{corollary}\label{cor:ppp-modulus}
Let $X$ be a Riesz space with the principal projection property and
let $Y$ be a Dedekind complete Riesz space. If $T\in\mathcal O_o(X,Y)$,
then $|T|$ is order continuous. Consequently,
$\mathcal O_o(X,Y)$ is a Riesz space.
\end{corollary}
\begin{proof}
By Proposition~\ref{prop:oc-popov}, $T\in P(X,Y)$. Hence
\cite[Theorem~6.1]{MykhaylyukPopov22} implies that $|T|$ is order
continuous. Since $\mathcal O_o(X,Y)$ is an ordered vector space and
is closed under taking moduli, it follows that $\mathcal O_o(X,Y)$
is a Riesz space.
\end{proof}

\begin{remark}
Mykhaylyuk and Popov posed the following problem in
\cite[Problem~4]{MykhaylyukPopov22}: does there exist a Riesz space
$X$ with the principal projection property (or even a Dedekind
complete Riesz space), a Dedekind complete Riesz space $Y$, and an
order continuous orthogonally additive operator $T\colon X\to Y$
such that $|T|$ is not order continuous?

The preceding corollary gives a negative answer to this problem.
Indeed, under precisely these assumptions, every order continuous
orthogonally additive operator $T\colon X\to Y$ has an order
continuous modulus.
\end{remark}

The following proposition shows that the principal projection property
in \cite[Proposition~1.3]{MykhaylyukPlievPopovSobchuk} can be weakened
to the existence of a cofinal family of projection bands.

\begin{proposition}\label{lem:weak-atom-cofinal}
Let $X$ be a Riesz space and let $u$ be a weak atom of $X$.
Denote by $B_u$ the principal band generated by $u$.
Then the following statements hold.
\begin{itemize}
\item[(i)] If $A$ is a nonzero projection band of $X$ such that
$A\subseteq B_u$, then $A=B_u$.

\item[(ii)] If $X$ has a cofinal family of projection bands,
then $u$ is an atom.
\end{itemize}
\end{proposition}

\begin{proof}
(i) Let $P_A$ be the band projection onto $A$.
Since $A$ is nonzero and $A\subseteq B_u$, we have
$u\notin A^d$, and hence $P_Au\ne0$.
Moreover, $P_Au\sqsubseteq u$, so $P_Au=u$,
since $u$ is a weak atom.
Thus $u\in A$, and consequently $B_u=A$.

(ii) Suppose that $X$ has a cofinal family of projection bands.
We now show that $u$ is an atom. Suppose, to the contrary,
that there exist nonzero $x,y\in X_+$ such that
\[
x,y\le |u|,\qquad x\perp y.
\]
By assumption, there exists a nonzero projection band
$E\subseteq B_x$. Since
\[
E\subseteq B_x\subseteq B_u,
\]
part (i) gives $E=B_u$, and hence $B_x=B_u$.
Thus $y\in B_x$. On the other hand, $x\perp y$
implies $y\in B_x^d$. Therefore
\[
y\in B_x\cap B_x^d=\{0\},
\]
a contradiction. Hence $u$ is an atom.
\end{proof}

We now turn to the order boundedness of order continuous orthogonally
additive operators. The following lemmas provide the groundwork for
our main theorem.

\begin{lemma}\label{lem:shrinking-intervals}
Let $X$ and $Y$ be Riesz spaces, and let $T:X\to Y$ be
an order continuous operator satisfying $T(0)=0$.
Suppose that
\[
x_\alpha\downarrow0.
\]
Then there exists $\alpha_0$ such that
\[
T([-x_{\alpha_0},x_{\alpha_0}])
\]
is order bounded in $Y$.
\end{lemma}

\begin{proof}
If the index set has a maximal element $\alpha_0$, then
$x_{\alpha_0}=0$. Since $T(0)=0$, the assertion follows.

Suppose that the index set $\Lambda$ has no maximal element. Set
\[
\Gamma
:=
\bigcup_{\alpha\in\Lambda}
\{(\alpha,x): |x|\leq x_\alpha\},
\]
and define
\[
(\alpha,x)\preceq(\beta,y)
\quad\Longleftrightarrow\quad
\alpha<\beta
\quad\text{or}\quad
(\alpha,x)=(\beta,y).
\]
Then $\Gamma$ is directed. For $\lambda=(\alpha,x)\in\Gamma$, put
$z_\lambda:=x$. 
Fix $\beta\in\Lambda$ and set
$\lambda_0=(\beta,0)$. For every
$\lambda=(\alpha,z)\succeq\lambda_0$, we have
$\alpha\ge\beta$, and hence
\[
|z_\lambda|\le x_\alpha\le x_{\beta}.
\]
It follows from the
definition of order convergence that
$z_\lambda\xrightarrow{o}0$.

Hence, by the order continuity of $T$,
\[
T(z_\lambda)\xrightarrow{o}0.
\]
Therefore $(T(z_\lambda))$ is eventually order bounded. Thus there
exist $(\alpha_1,x_1)\in\Gamma$ and $w\in Y_+$ such that
\[
|T(z_\lambda)|\leq w
\]
whenever $\lambda\succeq(\alpha_1,x_1)$. Choose
$\alpha_0>\alpha_1$. Then, for every $x$ satisfying
$|x|\leq x_{\alpha_0}$,
\[
(\alpha_1,x_1)\preceq(\alpha_0,x),
\]
and hence $|T(x)|\leq w$. Therefore
$T([-x_{\alpha_0},x_{\alpha_0}])$ is order bounded.
\end{proof}

\begin{lemma}\label{lem:real-order-bounded}
Let $Y$ be a Riesz space. Then every order continuous
mapping $f:\mathbb R\to Y$ is order bounded.
\end{lemma}

\begin{proof}
We first show that $f$ is locally order bounded.
Fix $\lambda_0\in\mathbb R$ and suppose, to the
contrary, that for every $\delta>0$ and
$w\in Y_+$ there exists $\lambda\in\mathbb R$
such that
\[
|\lambda-\lambda_0|<\delta,
\qquad |f(\lambda)|\nleq w.
\]

Define
\[
D=\left\{
(n,w,\lambda)\in
\mathbb N\times Y_+\times\mathbb R:
|\lambda-\lambda_0|<\frac1n,\,
|f(\lambda)|\nleq w
\right\}.
\]
Equip $D$ with the preorder
\[
(n,w,\lambda)\preceq(m,v,\mu)
\quad\Longleftrightarrow\quad
n\le m,\quad w\le v.
\]
Then $D$ is directed. Indeed, given
$d_1=(n_1,w_1,\lambda_1)$ and
$d_2=(n_2,w_2,\lambda_2)$, put
\[
n=\max\{n_1,n_2\},
\qquad w=w_1\vee w_2.
\]
By our assumption, there exists $\lambda$
such that $(n,w,\lambda)\in D$, and this
element is an upper bound of $d_1$ and $d_2$.

For $d=(n,w,\lambda)\in D$, set
$x_d=\lambda$. Since
\[
|x_d-\lambda_0|<\frac1n,
\]
and the first coordinates of the indices
eventually exceed every positive integer,
we obtain
\[
x_d\xrightarrow{o}\lambda_0.
\]
By order continuity,
\[
f(x_d)\xrightarrow{o}f(\lambda_0).
\]
Hence there exist $d_0\in D$ and $u\in Y_+$
such that
\[
|f(x_d)|\le u
\qquad(d\succeq d_0).
\]

Write $d_0=(n_0,w_0,\lambda')$ and put
$w=w_0\vee u$. By our assumption, there
exists $\lambda$ such that
\[
d=(n_0,w,\lambda)\in D.
\]
Since $d\succeq d_0$, we have
\[
|f(\lambda)|\le u\le w,
\]
contrary to the definition of $D$.
Thus $f$ is locally order bounded.

Now let $[a,b]$ be a compact interval.
Consider the family of all open intervals
$I\subseteq\mathbb R$ for which $f(I)$
is order bounded in $Y$.
By local order boundedness, this family
is an open cover of $[a,b]$.

By compactness, there exist finitely many
such intervals $I_1,\ldots,I_m$ satisfying
\[
[a,b]\subseteq\bigcup_{i=1}^m I_i.
\]
For each $i$, choose $u_i\in Y_+$ such that
\[
|f(\lambda)|\le u_i
\qquad(\lambda\in I_i).
\]
Then $u=u_1+\cdots+u_m$ satisfies
\[
|f(\lambda)|\le u
\qquad(\lambda\in[a,b]).
\]
Thus $f([a,b])$ is order bounded.

Every bounded interval is contained in
a compact interval, so the conclusion
follows.
\end{proof}

\begin{lemma}\label{lem:decreasing-fragments}
Let $X$ be a Riesz space with a cofinal family of projection bands,
and let $0<e\in X_+$. If $e$ has no nonzero atomic fragment, then
there exists a net $(e_\gamma)$ in
$\mathfrak F_e\setminus\{0\}$ such that
\[
e_\gamma\downarrow0.
\]
\end{lemma}

\begin{proof}

By \cite[Proposition 3.2]{MykhaylyukPlievPopov21},
$\mathfrak F_e$ is partially ordered by $\sqsubseteq$.
Since $e\ge0$, this order coincides with the order
inherited from $X$
(see \cite[Proposition 3.4(ii)]{MykhaylyukPlievPopov21}).
In particular, every element of $\mathfrak F_e$ is
positive, so $0$ is a lower bound of $\mathfrak F_e$.

By the Hausdorff maximal principle, the chain $\{e\}$
is contained in a maximal chain
\[
\mathcal C\subseteq\mathfrak F_e\setminus\{0\}.
\]
We claim that $\inf C=0$.
Suppose, to the contrary, that $C$ has a lower
bound $v$ with $v^+>0$. Since $0$ is also a
lower bound of $C$, the element $v^+$ is a
positive lower bound of $C$.

 Since
$X$ has a cofinal family of projection bands, there exists a nonzero
projection band $D\subseteq B_{v^+}$. Let $P_D$ be the corresponding
band projection and put
\[
d:=P_De.
\]
We have $d>0$. Indeed, if $P_Dv^+=0$, then $v^+\in D^d$, and hence
$B_{v^+}\subseteq D^d$, contradicting $0\ne D\subseteq B_{v^+}$. Since
$v^+\leq e$,
\[
0<P_Dv^+\leq P_De=d.
\]
Moreover, $d\sqsubseteq e$.

For every $c\in\mathcal C$, we have $v^+\leq c\sqsubseteq e$.
Hence $v^+\perp(e-c)$, and therefore
\[
e-c\in B_{v^+}\subseteq D^d.
\]
Consequently,
\[
d=P_De=P_Dc\leq c
\qquad(c\in\mathcal C).
\tag{1}
\]
Thus $d$ is comparable with every member of $\mathcal C$. By
maximality, $d\in\mathcal C$, and (1) shows that $d$ is the minimum
of $\mathcal C$.

By assumption, $d$ is not an atom. By
Proposition~\ref{lem:weak-atom-cofinal}, $d$ is not a weak atom.
Hence there exists
\[
b\in\mathfrak F_d\setminus\{0,d\}.
\]
Then $0<b<d\leq c$ for every $c\in\mathcal C$. Thus $b$ is
comparable with every member of $\mathcal C$, while
$b\notin\mathcal C$, contradicting the maximality of $\mathcal C$. Therefore $v^+ =0$ and so $v\le 0$. 

Thus, every lower bound $v$ of $C$ satisfies $v\le 0$.
Since $0$ is itself a lower bound of $C$,
we conclude that $\inf C=0$.

Regard $\mathcal C$ as a directed set with the reverse order and
define $e_c:=c$ for $c\in\mathcal C$. Then
\[
e_c\downarrow0,
\]
and every $e_c$ is a nonzero fragment of $e$.
\end{proof}

\begin{lemma}\label{lem:local-bounded-fragment-cofinal}
Let $X$ be a Riesz space with a cofinal family of projection bands,
let $Y$ be a Riesz space, and let $T:X\to Y$ be an order
continuous operator satisfying $T(0)=0$.
Then, for every $0<e\in X_+$, there exists
$0<u\sqsubseteq e$ such that
\[
T([-u,u])
\]
is order bounded in $Y$.
\end{lemma}

\begin{proof}
Fix $0<e\in X_+$.

\noindent\textbf{Case 1.}
Suppose that $e$ has a nonzero atomic fragment $a$.
Since $a$ is an atom,
\[
[-a,a]=\{\lambda a:-1\le\lambda\le1\}.
\]
Define $f:\mathbb R\to Y$ by
\[
f(\lambda)=T(\lambda a).
\]
Since $T$ is order continuous, so is $f$.
By Lemma~\ref{lem:real-order-bounded}, $f$ is
order bounded. Consequently,
\[
T([-a,a])=f([-1,1])
\]
is order bounded. Thus we may take $u=a$.

\medskip
\noindent\textbf{Case 2.}
Suppose that $e$ has no nonzero atomic fragment.
By Lemma~\ref{lem:decreasing-fragments} there exists a net
$(e_\gamma)$ in $\mathcal F_e\setminus\{0\}$ such that
\[
e_\gamma\downarrow0.
\]
By Lemma~3.5, there exists $\gamma_0$ such that
\[
T([-e_{\gamma_0},e_{\gamma_0}])
\]
is order bounded. Taking $u=e_{\gamma_0}$
completes the proof.
\end{proof}

\begin{theorem}\label{thm:order-continuous-oao-cofinal}
Let $X$ be a Riesz space with a cofinal family of projection bands
and let $Y$ be a Riesz space. Then every order continuous
orthogonally additive operator $T:X\to Y$ is order bounded.
In particular, $T$ is an abstract Uryson operator.
\end{theorem}

\begin{proof}
Fix $e\in X_+$. By
\cite[Proposition 3.4(ii)]{MykhaylyukPlievPopov21},
$\mathfrak F_e\subseteq X_+$. Set
\[
\mathcal A_e
:=
\{u\in\mathfrak F_e:
T([-u,u])\text{ is order bounded}\}.
\]
Clearly, $0\in\mathcal A_e$. Moreover, if
$u\in\mathcal A_e$ and $v\sqsubseteq u$, then
$0\le v\le u$, and hence
$[-v,v]\subseteq[-u,u]$.
Thus $v\in\mathcal A_e$.

\medskip
\noindent\textbf{Step 1.}

We show that $\mathcal A_e$ is directed.
Let $u,v\in\mathcal A_e$ and put
\[
r=(v-u)^+=v-(u\wedge v).
\]
Since $u,v\sqsubseteq e$, we have
\[
0\le r\le v,
\qquad
0\le r\le e-u.
\]
The first inequality implies that
$T([-r,r])$ is order bounded, and hence
$r\in\mathcal A_e$.
As $u\perp(e-u)$, the second inequality
yields $u\perp r$. Consequently,
\[
u\sqcup r=u+r=u\vee v\sqsubseteq e,
\]
where the last relation follows from
\cite[Proposition~3.4]{MykhaylyukPlievPopov21}.

For every $x$ with $|x|\le u\sqcup r$,
the Riesz decomposition property yields
\[
x=x_1+x_2,
\qquad
|x_1|\le u,\quad |x_2|\le r.
\]
Since $u\perp r$, we have $x_1\perp x_2$,
and therefore
\[
T(x)=T(x_1)+T(x_2).
\]
It follows that
$T([-u\sqcup r,u\sqcup r])$ is order bounded,
so $ u\vee v\in\mathcal A_e$.
Thus $\mathcal A_e$ is directed.

\medskip
\noindent\textbf{Step 2.}
For $u\in\mathcal A_e$, put
\[
b_u:=e-u.
\]
Then $(b_u)_{u\in\mathcal A_e}$ is decreasing.
We show that $b_u\downarrow0$.

Suppose, to the contrary, that some
$0<v\in X_+$ satisfies
\[
v\le b_u
\qquad(u\in\mathcal A_e).
\]
Choose a nonzero projection band
$D\subseteq B_v$, let $P_D$ be its
band projection, and put
\[
d:=P_De.
\]
Since  $v\notin D^d$,  $0< P_Dv \le d$.
Moreover, since $v\le b_u=e-u$, we have
$u\perp v$, and hence
$u\in B_v^d\subseteq D^d$. Therefore
\[
d=P_De=P_Db_u\le b_u
\qquad(u\in\mathcal A_e).
\tag{4}
\]

By Lemma~\ref{lem:local-bounded-fragment-cofinal},
there exists $0<w\sqsubseteq d$ such that
$T([-w,w])$ is order bounded.
Hence $w\in\mathcal A_e$.
Taking $u=w$ in (4) gives
\[
0<w\le d\le b_w=e-w.
\]
But $w\sqsubseteq e$, so
$w\perp(e-w)$, a contradiction.
Therefore $b_u\downarrow0$.

\medskip
\noindent\textbf{Step 3.}
By Lemma~\ref{lem:shrinking-intervals},
there exists $u_0\in\mathcal A_e$ such that
\[
T([-b_{u_0},b_{u_0}])
\]
is order bounded. Hence
$b_{u_0}\in\mathcal A_e$.
Since
\[
e=u_0\sqcup b_{u_0}.
\]

By the Riesz decomposition theorem and the
orthogonal additivity of $T$, we conclude that
$T([-e,e])$ is order bounded.

Finally, every order bounded subset of $X$
is contained in $[-e,e]$ for some $e\in X_+$.
Therefore $T$ is order bounded, and hence
is an abstract Uryson operator.
\end{proof}

Mykhaylyuk and Popov
\cite[Theorem~7.1]{MykhaylyukPopov22}
constructed an order continuous and order bounded
orthogonally additive operator on $C[0,1]$,
which does not have a cofinal family of projection
bands. Moreover, Abramovich and Sirotkin
\cite[Theorem~2.1]{AbramovichSirotkin05}
proved that every order continuous linear operator
between arbitrary Riesz spaces is automatically
order bounded, without any additional assumptions
on the domain. These results naturally raise the
question of whether the assumption on $X$ in
Theorem~\ref{thm:order-continuous-oao-cofinal}
can be omitted.

\begin{problem}
Can the assumption that $X$ has a cofinal family
of projection bands be removed from
Theorem~\ref{thm:order-continuous-oao-cofinal}?
In other words, is every order continuous
orthogonally additive operator between arbitrary
Riesz spaces necessarily order bounded?
\end{problem}

The proof of Theorem~\ref{thm:order-continuous-oao-cofinal}
relies on the Hausdorff maximal principle through
Lemma~\ref{lem:decreasing-fragments}, which is used
in Case~2 of the proof of
Lemma~\ref{lem:local-bounded-fragment-cofinal}.
This raises the following question.

\begin{problem}
Can Theorem~\ref{thm:order-continuous-oao-cofinal}
be proved in ZF, without assuming the axiom
of choice?
\end{problem}

We next show that the second question has an
affirmative answer when $X$ is an Archimedean
atomic Riesz space. Although atomicity is a
stronger structural assumption than the existence
of a cofinal family of projection bands, it allows
us to avoid the maximality argument used in
Lemma~\ref{lem:decreasing-fragments}.
In fact, the following result can be established
entirely within ZF.

\begin{proposition}\label{prop:atomic-order-bounded}
Let $X$ be an Archimedean atomic Riesz space and
let $Y$ be a Riesz space. Then every order continuous
orthogonally additive operator $T:X\to Y$ is
order bounded.
\end{proposition}

\begin{proof}
We first observe that $X$ has a cofinal family of
projection bands. Indeed, let $B$ be a nonzero band
in $X$. By atomicity, there exists an atom
$0<a\in B$. Since $X$ is Archimedean, the principal
band $B_a$ is a projection band. Thus
\[
\{0\}\subsetneq B_a\subseteq B,
\]
showing that $X$ has a cofinal family of
projection bands.

Fix $0<e\in X_+$ and define
\[
\mathcal A_e
:=
\{u\in\mathfrak F_e:
T([-u,u])\text{ is order bounded}\}.
\]
By Step~1 of the proof of
Theorem~\ref{thm:order-continuous-oao-cofinal},
$\mathcal A_e$ is directed.

For each $u\in\mathcal A_e$, put $b_u=e-u$.
Then $(b_u)_{u\in\mathcal A_e}$ is decreasing.
We show that $b_u\downarrow0$.

Suppose, to the contrary, that there exists
$0<v\in X_+$ such that
\[
v\le b_u \qquad(u\in\mathcal A_e).
\]
By atomicity, choose an atom $a$ with
$0<a\le v$. Since $B_a$ is a projection band,
put
\[
d:=P_{B_a}e.
\]
Then $d$ is a nonzero atomic fragment of $e$.
By Case~1 of the proof of
Lemma~\ref{lem:local-bounded-fragment-cofinal},
$T([-d,d])$ is order bounded. Hence
$d\in\mathcal A_e$.

On the other hand, $a\le v\le e-u$ for every
$u\in\mathcal A_e$. Since $u\sqsubseteq e$,
we have $u\perp(e-u)$, and therefore
\[
a\perp u \qquad(u\in\mathcal A_e).
\]
In particular, $a\perp d$. But
\[
0<a\le P_{B_a}e=d,
\]
a contradiction. Thus $b_u\downarrow0$.

By Lemma~\ref{lem:shrinking-intervals},
there exists $u_0\in\mathcal A_e$ such that
$T([-b_{u_0},b_{u_0}])$ is order bounded.

Since
\[
e=u_0\sqcup b_{u_0},
\]
as in Step~3 of the proof of
Theorem~\ref{thm:order-continuous-oao-cofinal},
the Riesz decomposition theorem and the
orthogonal additivity of $T$ imply that
$T([-e,e])$ is order bounded.

Since $e\in X_+$ was arbitrary, $T$ is
order bounded.
\end{proof}

Recall that every Dedekind complete Riesz space has a
cofinal family of projection bands. However, Dedekind
completeness and atomicity are independent properties.
For instance, $L^\infty[0,1]$ is Dedekind complete but
not atomic, whereas the Riesz space $c$ of convergent
real sequences is atomic but not Dedekind complete.

We next show that, when $X$ is Dedekind complete,
every order continuous orthogonally additive operator
$T:X\to Y$ is order bounded. Moreover, this result
can be established in ZF together with the axiom
of dependent choice (DC), without invoking the
full axiom of choice.

\begin{theorem}\label{thm:dedekind-complete-DC}
Let $X$ be a Dedekind complete Riesz space and
let $Y$ be a Riesz space. Then every order continuous
orthogonally additive operator $T:X\to Y$ is
order bounded.
\end{theorem}

\begin{proof}
Fix $0<e\in X_+$. We first show that there exists
a nonzero fragment $u\sqsubseteq e$ such that
$T([-u,u])$ is order bounded.

\medskip
\noindent\textbf{Case 1.}
Suppose that $e$ has a nonzero atomic fragment $a$.
By Case~1 of the proof of
Lemma~\ref{lem:local-bounded-fragment-cofinal},
$T([-a,a])$ is order bounded.

\medskip
\noindent\textbf{Case 2.}
Suppose that $e$ has no nonzero atomic fragment.
Since $X$ is Dedekind complete, every band is a
projection band. In particular, every weak atom
is an atom. Thus no nonzero fragment of $e$
is a weak atom.

Set $x_1=e$. Since $x_1$ is not a weak atom,
there exists a nonzero proper fragment
$x_2\sqsubseteq x_1$. Similarly, since $x_2$
is a nonzero fragment of $e$ and hence is not
a weak atom, we may choose a nonzero proper
fragment $x_3\sqsubseteq x_2$.
Continuing inductively, using DC, we obtain
a sequence $(x_n)$ such that
\[
0<x_{n+1}\sqsubseteq x_n,
\qquad x_{n+1}\ne x_n
\quad(n\ge1).
\]

By Dedekind completeness, the infimum
$x=\inf_{n\ge1}x_n$ exists. Since
$\mathfrak F_e$ is closed under existing
infima, $x\sqsubseteq e$ and
$x\sqsubseteq x_n$ for every $n$.

Put $y_n=x_n-x$. Then
\[
0<y_n\sqsubseteq e,
\qquad y_n\downarrow0.
\]
Indeed, $y_n=0$ would imply
$x_n=x\le x_{n+1}<x_n$, a contradiction.

By Lemma~\ref{lem:shrinking-intervals},
there exists $n_0$ such that
$T([-y_{n_0},y_{n_0}])$ is order bounded.

Thus, in either case, every nonzero positive
element of $X$ has a nonzero fragment on whose
symmetric order interval $T$ is order bounded.

\medskip
Now define
\[
\mathcal A_e
:=
\{u\in\mathfrak F_e:
T([-u,u])\text{ is order bounded}\}.
\]
By Step~1 of the proof of
Theorem~\ref{thm:order-continuous-oao-cofinal},
$\mathcal A_e$ is directed.

For each $u\in\mathcal A_e$, put $b_u=e-u$.
Since $\mathcal A_e$ is directed, the net
$(b_u)_{u\in\mathcal A_e}$ is decreasing.
By Dedekind completeness, its infimum
\[
b:=\inf_{u\in\mathcal A_e}b_u
\]
exists.

We claim that $b=0$. Suppose otherwise.
Since every $b_u$ is a fragment of $e$,
so is $b$. By the local boundedness property
established above, there exists a nonzero
fragment $v\sqsubseteq b$ such that
$T([-v,v])$ is order bounded. Hence
$v\in\mathcal A_e$.

For every $u\in\mathcal A_e$, we have
$v\le b\le e-u$, so $u\perp v$.
Hence $u+v=u\vee v\sqsubseteq e$.
By the Riesz decomposition theorem and the
orthogonal additivity of $T$,
$T([-u-v,u+v])$ is order bounded.
Thus $u+v\in\mathcal A_e$.

Since $b$ is a lower bound of $(b_u)$,
\[
b\le e-(u+v)=b_u-v
\]
for every $u\in\mathcal A_e$.
Taking the infimum over $u$ yields
$b\le b-v$, contradicting $v>0$.
Therefore
\[
b_u=e-u\downarrow0.
\]

By Lemma~\ref{lem:shrinking-intervals},
there exists $u_0\in\mathcal A_e$ such that
$T([-b_{u_0},b_{u_0}])$ is order bounded.

Since
\[
e=u_0\sqcup b_{u_0},
\]
as in Step~3 of the proof of
Theorem~\ref{thm:order-continuous-oao-cofinal},
the Riesz decomposition theorem and the
orthogonal additivity of $T$ imply that
$T([-e,e])$ is order bounded.

Since $e\in X_+$ was arbitrary, $T$ is
order bounded.
\end{proof}

Fotiy, Kadets and Popov characterized the Riesz spaces $X$ and $Y$
for which every orthogonally additive operator from $X$ to $Y$ is
laterally-to-order bounded; see
\cite[Theorem~2.1]{FotiyKadetsPopov23}. It is therefore natural to
ask when every orthogonally additive operator is order bounded.
For Riesz spaces with a cofinal family of projection bands, the
following result gives a complete answer.

\begin{theorem}
Let $X$ and $Y$ be Riesz spaces such that $X$ has a cofinal family
of projection bands and $Y$ is Archimedean. Then
\[
O(X,Y)=U(X,Y)
\]
if and only if $X=\{0\}$ or $Y=\{0\}$.
\end{theorem}

\begin{proof}
The ``if'' part is immediate. For the converse, suppose that
$O(X,Y)=U(X,Y)$ and assume that $X\neq\{0\}$ and $Y\neq\{0\}$.
Since
\[
U(X,Y)\subseteq P(X,Y)\subseteq O(X,Y),
\]
it follows that $P(X,Y)=O(X,Y)$. Hence, by
\cite[Theorem~2.1]{FotiyKadetsPopov23}, the set $\mathfrak F_e$ is
finite for every $e\in X$.

Choose $0<e\in X$. Since $\mathfrak F_e\setminus\{0\}$ is a nonempty
finite partially ordered set, it has a minimal element, say $a$.
Recall that, for $x,y\in\mathfrak F_e$,
\[
x\leq y \quad\Longleftrightarrow\quad x\sqsubseteq y
\]
(see \cite[Proposition~3.4(2)]{MykhaylyukPlievPopov21}).
We claim that $\mathfrak F_a=\{0,a\}$. Indeed, if
$0\neq b\in\mathfrak F_a$, then $b\sqsubseteq a\sqsubseteq e$, and
hence $b\in\mathfrak F_e$. Moreover, $b\leq a$. By the minimality of
$a$ in $\mathfrak F_e\setminus\{0\}$, we obtain $b=a$. Thus $a$ is a
weak atom. By Proposition~\ref{lem:weak-atom-cofinal}, $a$ is an atom and
the principal band $B_a$ generated by $a$ is a projection band.

Since $a$ is an atom,
\[
B_a=\operatorname{span}\{a\}.
\]
Let $P_a:X\to B_a$ be the band projection onto $B_a$. Thus, for every
$x\in X$, there exists a unique scalar $\lambda_x\in\mathbb R$ such that
\[
P_a x=\lambda_x a.
\]

Choose $0<y\in Y$ and a function $g:\mathbb R\to\mathbb R$ such that
$g(0)=0$ and $g$ is unbounded above on $[0,1]$. Define
\[
T(x)=g(\lambda_x)y,\qquad x\in X.
\]
We show that $T$ is orthogonally additive. Let $x,z\in X$ with
$x\perp z$. Since the band projection $P_a$ preserves disjointness,
$P_a x\perp P_a z$. Both $P_a x$ and $P_a z$ belong to the
one-dimensional band $B_a$, so at least one of them is zero. Hence
either $\lambda_x=0$ or $\lambda_z=0$. Since $P_a$ is linear,
$\lambda_{x+z}=\lambda_x+\lambda_z$, and therefore, using $g(0)=0$,
\[
T(x+z)
 =g(\lambda_x+\lambda_z)y
 =g(\lambda_x)y+g(\lambda_z)y
 =T(x)+T(z).
\]
Thus $T\in O(X,Y)$.

On the other hand, for every $\lambda\in[0,1]$ we have
$\lambda a\in[0,a]$ and
\[
T(\lambda a)=g(\lambda)y.
\]
Since $g$ is unbounded above on $[0,1]$ and $0<y\in Y$, the
Archimedean property of $Y$ implies that
\[
\{g(\lambda)y:\lambda\in[0,1]\}
\]
is not order bounded. Hence $T([0,a])$ is not order bounded, so
$T\notin U(X,Y)$. This contradicts $O(X,Y)=U(X,Y)$.
Therefore $X=\{0\}$ or $Y=\{0\}$.
\end{proof}

\section*{Declarations}

\noindent\textbf{Conflict of interest.}
The author declares that there is no conflict of interest.

\medskip

\noindent\textbf{Data availability.}
No datasets were generated or analysed during the current study.

\medskip

\noindent\textbf{Funding.}
The author received no funding for this work.

\end{document}